\documentclass[11pt,a4paper]{amsart}

\usepackage{amssymb,amsmath,epsfig,graphics,mathrsfs}
\usepackage{graphicx}

\usepackage{fancyhdr}
\usepackage{hyperref}
\hypersetup{
 colorlinks   = true,
 urlcolor     = blue,
 linkcolor    = blue,
 citecolor   = red ,
 bookmarksopen=true
}

\newtheorem{thm}{Theorem}[section]
\theoremstyle{plain}
\newtheorem{lem}[thm]{Lemma}

\theoremstyle{definition}
\newtheorem{defi}{Definition}[section]
\newtheorem{rem}[thm]{Remark}

\newcommand{\rN}{\mathbb{R}^{N}}

\newcommand{\ro}{\mathbb{R}}
\newcommand{\Om}{\Omega}

\newcommand{\mcal}[1]{\mathcal{#1}}
\newcommand{\md}[1]{\left|#1 \right|}
\newcommand{\nrm}[1]{\left|\left| #1 \right|\right|}
\newcommand{\bct}[1]{\left(#1\right)}

\newcommand{\ep}{\epsilon}

\numberwithin{equation}{section} \allowdisplaybreaks

         \title[On the Strong unique continuation Property]{On the Strong unique continuation property of   a  degenerate elliptic operator with Hardy type potential}
\author{Agnid Banerjee, Arka Mallick}
\address{TIFR  Centre for Applicable Mathematics, Post Bag No. 6503
 Sharadanagar, \\ Bangalore 560065, India.}
\email{arkamallick02@gmail.com, agnid@math.tifrbng.res.in} 

\begin{document}

\maketitle


\begin{abstract}
In this paper  we prove  the  strong unique continuation property  for the following degenerate elliptic equation 
\begin{equation}\label{e0}
\Delta_zu +|z|^2\partial_t^2u = Vu,\quad  (z,t) \in \mathbb{R}^N \times \mathbb{R}
\end{equation}
where the potential $V$ satisfies  either of the following growth assumptions
\begin{align}
& \md{V(z,t)} \leq  \frac{f(\rho(z,t))}{\rho(z,t)^2},\ \text{where $\rho$ is as in \eqref{gauge} and $f$ satisfies the Dini integrability condition as in \eqref{Dini}.}
\\
&  \text{or when }
\notag
\\
& \md{V(z,t)} \leq C \frac{\psi(z,t)^{\epsilon}}{\rho(z,t)^2},\  \text{for some $\epsilon>0$ with $\psi$ as in \eqref{gweight} and  $N$  even.}
\notag
\end{align}
This extends some of the previous results obtained in \cite{G} for this subfamily of Baouendi-Grushin operators. As corollaries, we obtain new unique  continuation properties for solutions $u$  to
\[
\Delta_{\mathbb{H}} u = Vu
\]
with certain symmetries as expressed  in \eqref{t} where $\Delta_{\mathbb{H}}$ corresponds to  the sub-Laplacian on  the Heisenberg group $\mathbb{H}^n$. 

\end{abstract}


\section{Introduction}
\noindent An operator $L$  (local or non-local) is said to possess the  strong unique continuation property in the $L^{p}$ sense if  any non-trivial solution $u$ to
\[
Lu=0
\]
in a (connected) domain $\Om \subset \mathbb{R}^n$ cannot vanish to infinite order in the $L^{p}$ mean at any point in $\Om$.  We refer to  Definition \ref{iv} for the precise notion of vanishing to infinite order in the  $L^{p}$ mean.

\medskip

\noindent The fundamental role played by strong unique continuation theorems in the theory of partial differential equations is well known.  We recall that in his foundational paper in 1939( see \cite{C}), T. Carleman established the  strong unique continuation property for
\[
\Delta u = Vu
\]
in $\mathbb{R}^2$ under the assumption that $V$ is in $L^{\infty}$.  This result was subsequently extended by several mathematicians to arbitrary dimension and also to  equations with variable coefficients.  In this direction, we refer to the pioneering work of Aronszajn-Krzywicki-Szarski (see \cite{AKS}), where strong unique continuation property for elliptic operators  with Lipschitz principal part   was established by using generalization  of the estimates of Carleman.  About twenty years later a different geometric approach independent of the Carleman estimates came up in the  seminal works of Garofalo-Lin  in 1986, see \cite{GL1}, \cite{GL2}. Their method, which is based on the almost monotonicity of a generalized frequency function, allowed them to obtain new quantitative information on the zero set of solutions to divergence form elliptic equations and, in particular, encompassed the  results of \cite{AKS}. The reader should note that  such a  frequency  function approach  has its roots in the well-known work of Almgren \cite{Al} which plays a crucial role in the regularity theory of mass minimizing currents. 

\medskip

 There has also been an interest of mathematicians working  in partial differential equations and mathematical physics to focus on the unique continuation property for equations with unbounded lower order terms. Subsequent developments in this direction have  culminated with Jerison and Kenig's celebrated result on the strong unique continuation property for
 \[
 \Delta u= Vu
 \]
 with $V \in L^{n/2}_{loc}(\mathbb{R}^n)$, see \cite{JK}. 
  Their paper has inspired much progress in the subject and nowadays the picture for second order uniformly elliptic equations is almost complete. See for instance \cite{KT} where analogous results have been established for variable coefficient  operators with Lipschitz  principal part. We also  refer to the work of Y. Pan \cite{Pa}, where strong unique continuation property for elliptic equations with scaling critical Hardy type potentials has been established, as well as  to that of Chanillo-Sawyer  ( see \cite{ChS})  for  strong  unique continuation results with   potentials in the   Fefferman-Phong class. 
  
  \medskip
  
  On the contrary, not so well understood is the situation concerning sub-elliptic operators.  It turns out that unique continuation is generically not  true in such context. This follows from a counterexample due to Bahouri \cite{Bah}, where the author showed that unique continuation is not true for  even smooth and  compactly supported perturbations of the sub-Laplacian on the Heisenberg group $\mathbb{H}^n$.  The first positive result in this direction came up in the work of Garofalo-Lanconelli \cite{GLa} where among other important results, the authors showed that strong unique continuation  result holds  for solutions $u$ to 
  \[
  \Delta_{\mathbb{H}} u= Vu
  \]
  when $u$ has certain symmetries  as expressed in \eqref{t} and  with certain  growth conditions on $V$, see Theorem \ref{gl} below. Here, $\Delta_{\mathbb{H}}$  denotes the standard  sub-Laplacian on  $\mathbb{H}^n$.   We also  refer to the papers \cite{GR} and \cite{B} for the extension of the unique continuation  result in \cite{GLa}  to  Carnot groups of arbitrary step with appropriate  symmetry assumptions on  the solution  $u$ which can be thought of as  analogous to that in \eqref{t}.  It is to be noted that the results in \cite{B}, \cite{GLa} and \cite{GR} follow the circle of ideas as in the fundamental works \cite{GL1} and \cite{GL2} based on Almgren type frequency function approach.

  \medskip
  
  In a somewhat related direction,  the study of strong unique continuation  for zero order perturbations of the following degenerate  Baouendi-Grushin type operators 
  \[
  \mathcal{B}_\beta u= \Delta_{z} u + \frac{|z|^{2\beta}}{4} \Delta_{t} u,\ z \in \mathbb{R}^N, \  t \in \mathbb{R}^m,\ \beta>0
  \]
  was initiated by Garofalo in \cite{G},  where the author introduced an Almgren type  frequency function associated with $\mathcal{B}_\beta$  and proved that such a  frequency is bounded for solutions $u$  to
  \begin{equation}\label{b}
  \mathcal{B}_\beta u= Vu
  \end{equation}
  when $V$  satisfies  the following growth assumption
  \begin{equation}\label{growth}
  |V(z,t) |\leq C \frac{f(\rho(z,t))}{\rho(z,t)^2} \psi(z, t),
  \end{equation}
   for some  non-decreasing $f$ which is Dini integrable, i.e.
  \begin{equation}\label{Dini}
  \int_{0}^{R_0}\frac{f(r)}{r} < \infty,\ \text{for some $R_0>0$}
  \end{equation} and  where 
 \[
 \rho(z,t)= (|z|^{2(\beta+1)} + (\beta+1)^2 |t|^2)^{\frac{1}{2(\beta+1)}}\ \ \ \ \text{and}\ \ \ \psi(z,t)=\frac{|z|^{2\beta}}{\rho(z,t)^{2\beta}}, 
 \]
 
  Consequently  using the boundedness of such a frequency,  the author inferred that the $L^{2}$ doubling property of solutions to \eqref{b}  follows which in particular implies the strong unique continuation property. Also the case when  $V$ satisfies 
  \begin{equation}\label{gr2}
  0 < V^+(z,t) \leq C \frac{\psi}{\rho(z,t)^2},\\\ 0 \leq V^-(z,t) \leq \delta \frac{\psi}{\rho(z,t)^2},
  \end{equation}
  for some  $\delta>0$ small enough depending on $N, m$ and $\beta$ was studied  in the very same paper and in which case a slightly weaker version of unique continuation property was established( See Theorem 4.4 in \cite{G}). 
The results in \cite{G} were subsequently extended in \cite{GV} and \cite{B} to variable coefficient Baouendi-Grushin operators.   Note that the weight $\psi$ in \eqref{growth} /\eqref{gr2} degenerates on the submanifold $\{z=0\}$ and so the result in \cite{G} doesn't allow to take $V \in L^{\infty}$. The natural appearance of this degenerate weight $\psi$  can be seen from the fact that if
  \[
  u= f(\rho(z, t))
  \]
  then 
  \[
  Lu= \psi \left( f''(\rho) + \frac{Q-1}{\rho} f'(\rho) \right),\  \text{where $Q= N+ (\beta+1)m$}.
  \]
When $\beta=1, m=1$,  we note that the operator $\mathcal{B}_\beta$ is intimately connected to the sub-Laplacian on $\mathbb{H}^{n}$.  In order to see such a connection, we note that  with respect to the standard coordinates $(z, t)= (x, y, t)$ on $\mathbb{H}^n$( $=$ $\mathbb{R}^{n} \times \mathbb{R}^n \times \mathbb{R})$, the sub-Laplacian is given by
  \begin{equation}
  \Delta_{\mathbb{H}} = \Delta_z + \frac{|z|^2}{4} \partial_{t}^2 + \sum_{i=1}^n \partial_t (y_j \partial_{x_j} - x_j \partial_{y_j})
  \end{equation}
  Now if $u$ is a smooth function that annihilates the vector field  $T=\sum_{i=1}^n  (y_j \partial_{x_j} - x_j \partial_{y_j})$ (i.e., $Tu \equiv 0$), then  ( upto a normalization factor of $4$)  we have that 
  \[
  \Delta_{\mathbb{H}} u= \mathcal{B}_1 u,\ \text{for $m=1$}. 
  \]
  Note that it is not difficult to recognize that 
  \begin{equation}\label{t}
  Tu=0\ \text{iff}\ u(e^{i\theta} z, t)= u(z, t).
  \end{equation}
  Said differently, $Tu=0$ if and only if $u$ is invariant with respect to the natural action of the torus on $\mathbb{H}^n$.  
  \medskip
  
  In this  very same case (i..e when $\beta=1$ and $m=1$),  by establishing  very delicate $L^{p}-L^{q}$  Carleman estimates  as stated  in \eqref{LpLqCarEst}, Garofalo and Shen in \cite{GarShen}   obtained  strong unique continuation for \eqref{e0} when $V \in L^{r}_{loc}(\mathbb{R}^{N+1})$ for  $r> N=Q-2$  when $N$ is even and $r>\frac{2N^2}{N+1}$  when $N$ is odd and hence succeeded in removing the degenerate weight $\psi$ from their  growth assumption  unlike  that in \eqref{growth} or \eqref{gr2} for this subfamily of $\{\mathcal{B}_\beta\}$.  The reader should note that the approach in \cite{GarShen} in turn is inspired by  that  of Jerison's work in \cite{J} where a simpler proof of the Jerison-Kenig's Carleman inequality was discovered. 
  
  \medskip
  
  Now for a historical account,  we   recall that a more general class of operators  modelled on $\mathcal{B}_\beta$ was first introduced by Baouendi who studied the Dirichlet problem in weighted Sobolev spaces ( see \cite{Ba}). Subsequently, Grushin in \cite{Gr1} and \cite{Gr2} studied the hypoellipticity of the operator $\mathcal{B}_\beta$ when $\beta \in \mathbb{N}$ and showed that this property is however affected by the addition of lower order terms. We  would also like to  refer to the papers \cite{FGW} and \cite{FL} for H\"older regularity of weak solutions to    general equations modelled on $\mathcal{B}_{\beta}$.  Remarkably,  the operator of Baouendi also played an important role in the recent work \cite{KPS} on the higher regularity of the free boundary in the classical Signorini problem.   We would also like to mention that a version of the Almgren type monotonicity formula for $\mathcal{B}_\beta$ played an extensive role  in  \cite{CSS} on the obstacle problem for the fractional Laplacian.    \medskip
  
  Therefore  given the relevance of these sub-ellliptic operators in various contexts,  in this paper we study the strong unique continuation for  zero order perturbations of the operator $\mathcal{B}_1$ (when $m=1$) with singular potentials of Hardy type.    We first show  that  quite interestingly,   by employing   the $L^{p}-L^{q}$ type Carleman estimate derived in \cite{GarShen}, one can prove strong unique continuation property for \eqref{e0} (see Theorem \ref{sucpGr} below) when the potential $V$ satisfies the following  growth condition
\begin{align}\label{conPot}
\md{V(z,t)} \leq  \frac{f(\rho(z,t))}{\rho(z,t)^2} \text{ for a.e.  $(z,t) \in B_{R_0}$ and $f$ satisfies   \eqref{Dini}},
\end{align}  
i.e. we show that in the growth condition \eqref{growth} as  treated in \cite{G}, the degenerate weight $\psi$ can be removed from the growth assumption for this subfamily of operators for the validity of strong unique continuation. Typical representatives of $f$ satisyfing \eqref{Dini} are $f(r)=Cr^{\epsilon}$ ($\epsilon>0$) or $f(r)= C|\text{log}(1/r)|^{-\alpha}$ ($\alpha>1$) and therefore the growth assumption in \eqref{conPot} can be thought of as an ``almost" Hardy type growth  condition.

\medskip

Then by using the $L^{2}$  estimates for   projection operators $P_k$  as established in \cite{GarShen}( (see \eqref{projOp} for the definition of $P_k$) , we   establish a certain $L^{2}-L^{2}$ type Carleman estimate  where we obtain a particular  asymptotic behaviour of the constant involved in terms of  a parameter $s$ which corresponds to the exponent of the singular weight in the Carleman inequality ( See the estimate in  Theorem \ref{L2L2CarEst} below).   Using such an asymptotic behaviour of the constant,  we adapt   an  argument     in \cite{Pa} to our sub-elliptic setting and consequently obtain  an    analogous strong unique continuation  result  for   equation \eqref{e0} when the potential $V$ satisfies the following Hardy type growth condition

 \begin{align}\label{conPot1}
\md{V(z,t)} \leq C \frac{\psi(z,t)^{\epsilon}}{\rho(z,t)^2} \text{ for a.e.  $(z,t) \in B_{R_0}$,  $\epsilon >0$ and $N$  even,}
\end{align}
which again improves the growth assumption in \eqref{gr2} for this subclass  of Bauoendi-Grushin operators (see Theorem \ref{sucp2} below and also Theorem \ref{sucp4} for the corresponding result when $N$ is odd). The reader should note that our result Theorem \ref{sucp2} is new  even for $\epsilon=1$ because the growth assumption in \eqref{gr2} requires  $\delta$ to be sufficiently  small.    We would also like to mention that although we closely follow the approach  of \cite{Pa} in parts, it has nonetheless required some delicate modifications  in our setting  as the reader can see in the proof of Theorem \ref{sucp2}  in Section 4.  This is essentially due to the presence of the degenerate weight $\psi$   in our Carleman estimates. We  also note that  a generic potential $V$  satisfying our growth conditions  in \eqref{conPot} or \eqref{conPot1} need not be in $L^{r}$  for the range of $r$ dealt in  \cite{GarShen} and therefore  the class of potentials covered in our work are somewhat complementary to that treated in  \cite{GarShen}.   
\medskip

A few  open problems as well as remarks   are listed  in order:
\begin{itemize}
\item[1)] We note that  in  our growth condition \eqref{conPot1}, the parameter  $\epsilon$ which corresponds to the exponent of $\psi$ can be taken  arbitrarily small. Said differently,  we show in our Theorem \ref{sucp2}  that  the degenerate weight  $\psi$ can be ``almost" removed in the  Hardy type growth assumption.   It however  remains to be seen whether in \eqref{conPot1},  one  can get rid of  the  degenerate weight $\psi$ completely from the growth condition, i.e.  whether   one can take $\epsilon=0$ in \eqref{conPot1}  so that Theorem \ref{sucp2} continues to hold.  This appears  to be a challenging open problem to which we would like to come back in a future study.   
\item[2)] It also appears  interesting to   look at generalization of our unique continuation results as well as that of \cite{GarShen}  for the case when $m>1$, i.e. for  equations of the type
\begin{equation}\label{e1}
 \Delta_z u + |z|^2 \Delta_{t} u = Vu,\ (z, t) \in \mathbb{R}^N \times \mathbb{R}^m
\end{equation}
 The reader should note that this  would consequently have  similar applications to unique continuation properties for  sub-Laplacian type equations  on $H$- type groups ( see Section 9 in \cite{GR} for a detailed account on this aspect).  This seems to be a   challenging  issue as well  where one would need  to establish  estimates similar to that stated in Theorem \ref{wtdL2est}   for   appropriate projection operators.
\item[3)] It would also be interesting to look at  backward uniqueness results for zero order perturbations of  the  parabolic counterpart of the operators  as in \eqref{e0} or  more generally of the ones  as in \eqref{e1}.  We refer the reader to the interesting papers \cite{Po}, \cite{E}, \cite{EV} and  \cite{ESS} for the corresponding results  in the case of  uniformly parabolic  operators. 
\end{itemize}

\medskip

The paper is organized as follows. In Section 2, we introduce certain relevant notions and gather some known results. In Section 3, we establish a certain  $L^{2}-L^{2}$ type Carleman estimate  with a particular asymptotic behavior of the corresponding constant  as mentioned above using which we prove our strong unique continuation result Theorem \ref{sucp2}.  In Section 4, we  prove our main results Theorem \ref{sucpGr} and Theorem \ref{sucp2} using the Carleman estimates in Section 3.  Finally in Section 5, we  show application of our results  to a new  unique continuation property for stationary Schr\"odinger equations on the Heisenberg group $\mathbb{H}^n$.

\medskip

\textbf{Acknowledgment:} One of us, A.B.  would like to thank   Prof. Nicola Garofalo for several  suggestions and feedback related to this work.

\section{Preliminaries} 
The content of this section is essentially borrowed from \cite{GarShen}. The main goal is to introduce a suitable polar coordinates with respect to the non isotropic gauge function defined in \eqref{gauge} below and show how the Grushin operator interacts with these newly introduced polar coordinates.

\begin{align}\label{gauge}
\rho(z,t) = \bct{|z|^4+4t^2}^\frac{1}{4}, \  z\in \rN, \ t\in \ro.
\end{align}

For $0<\phi <\pi$, $0<\theta_i< \pi$, $i= 1,2, \dots, N-2$ and $0<\theta_{N-1}<2\pi$ let

\begin{align}\label{pcrdnts}
\begin{cases}
z_1= \rho \operatorname{sin}^\frac{1}{2}\phi \operatorname{sin}\theta_1\dots \operatorname{sin} \theta_{N-2}\operatorname{sin} \theta_{N-1} \\ z_2 = \rho \operatorname{sin}^\frac{1}{2}\phi \operatorname{sin}\theta_1\dots \operatorname{sin} \theta_{N-2}\operatorname{cos} \theta_{N-1} \\ \vdots \\ z_N = \rho \operatorname{sin} ^\frac{1}{2} \phi \operatorname{cos} \theta_1 \\ t= \frac{\rho^2}{2}\operatorname{cos}\phi.
\end{cases}
\end{align}

Then as computed in \cite{GarShen},  we have
\begin{align}\label{pmeasure}
r &= |z|= \rho \operatorname{sin} ^\frac{1}{2} \phi, \notag\\ 
dz dt & = \frac{1}{2} \rho^{N+1} \bct{\operatorname{sin} \phi}^\frac{N-2}{2} d\rho d\phi d\omega  \\
\mcal{L} &= \Delta_z +|z|^2\frac{\partial^2}{\partial t^2}= \sin \phi \bct{\frac{\partial^2}{\partial\rho^2} + \frac{N+1}{\rho} \frac{\partial}{\partial \rho} + \frac{4}{\rho^2}\mcal{L}_\sigma},\notag
\end{align}
where $d\omega$ denotes the surface measure on $S^{N-1}$ and 
\begin{align}\label{GrBeltrami}
\mcal{L}_\sigma = \frac{\partial^2}{\partial \phi^2} +\frac{N}{2} \frac{\cos\phi}{\sin\phi} \frac{\partial}{\partial \phi} + \frac{1}{\bct{2\sin \phi}^2} \Delta_{S^{N-1}}. 
\end{align}
Here $\sigma = (\phi, \omega)$, $\omega \in S^{N-1}$ and $\Delta_{S^{N-1}}$ denotes the Laplace-Beltrami operator on $S^{N-1}$. Notice that

\begin{equation}\label{iden}
 \operatorname{sin} \phi = \psi
 \end{equation}
where 
\begin{align}\label{gweight}
\psi(z,t) := \frac{r^2}{\rho^2}=\frac{|z|^2}{\bct{|z|^4+4t^2}^\frac{1}{2}}
\end{align}

The following lemma characterizes the spherical Harmonics for the Grushin operator.
\begin{lem}\label{spHar}
Let $k$ be nonnegative integer and $l= k(\mod 2)$, with $0\leq l \leq k$. Suppose that $Y_l$ is a spherical harmonic of degree $l$ for $\Delta_{S^{N-1}}$. Then 
\begin{align*}
g(\phi,\omega) = \sin^{\frac{l}{2}}  \phi C_{\frac{k-l}{2}}^{\frac{l}{2}+\frac{N}{4}}\bct{\cos \phi} Y_l(\omega)
\end{align*}
satisfies $\mcal{L}_\sigma g = -\frac{k(N+k)}{4}g$. Here, $C_{\frac{k-l}{2}}^{\frac{l}{2}+\frac{N}{4}}(\tau)$ is an ultraspherical  (or Gegenbauer) polynomial. We refer to  page 174 in  [E]  for relevant details. 
\end{lem}
As in \cite{GarShen}, we now define,

\begin{align}\label{projSpc}
\mcal{H}_k = \operatorname{span}\left\{ \sin^{\frac{l}{2}}  \phi C_{\frac{k-l}{2}}^{\frac{l}{2}+\frac{N}{4}}\bct{\cos \phi} Y_{l,j}(\omega) \big|  j = 1,2,\dots, d_l, \ 0\leq l\leq k, \ l= k(\mod2)\right\},
\end{align}
where $d_l = \frac{(N+2l-2)\Gamma(N+l-2)}{\Gamma(l+1) \Gamma(N-1)}$ and $\{Y_{l,j}\}_{j=1, ..., d_l}$ denote an  orthonormal basis for the space of spherical harmonics of degree $l$ on $S^{n-1}$. 
By taking $\rho = 1$ in \eqref{pcrdnts}, we can parametrize 

\begin{align}\label{gaugeSp}
\Om = \left\{ (z,t) \in \ro^{N+1} \big| \rho(z,t) = \bct{|z|^4+4t^2}^\frac{1}{4}=1\right\}
\end{align}
and consider the measure given by,
\begin{align}\label{Spmeasure}
d\Om = \sin^\frac{N}{2} d\phi d\omega.
\end{align}
Then as shown in Lemma 2.11 of \cite{GarShen}, we have that
\begin{equation}\label{direct}
L^2(\Om,d\Om) = \bigoplus_{k=0}^\infty \mcal{H}_k.
\end{equation}
 We  also let 
\begin{align}\label{projOp}
P_k: L^2(\Om,d\Om) \rightarrow \mcal{H}_k
\end{align}
 be  the projection operator onto the $(k+1)-$th eigenspace of $\mcal{L}_\sigma$ defined in \eqref{GrBeltrami}. 
 
 \medskip

 \textbf{The Function space $M^{2,2}(\Om)$:}
 In order to introduce  the notion of solution  for the equation   \eqref{e0}, it is natural to consider the following  function space associated with the H\"ormander vector fields
 \begin{align}
 & X_i= \partial_{z_i},\  i=1,...,N. 
 \\
 &Y_j = z_j \partial_t,\ j=1, ...,N.
 \notag
 \end{align}
 We define 
 \[
 M^{2,2}(\Om)=\{ u \in L^{2}(\Om): X_i u,Y_j u, X_i X_j u, X_i Y_j u, Y_iX_ju,Y_i Y_j u \in L^{2}(\Om), i,j \in \{1,\dots, N\}\}.
 \]
 Note that, by Theorem 1 in \cite{Xu} $, M^{2,2}(\Om)$ forms a Hilbert space with respect to the norm
 \begin{align}\label{m22nrm}
 \nrm{u}_{M^{2,2}(\Om)} : = &\nrm{u}_{L^2(\Om)} + \nrm{\md{\nabla_zu}^2}_{L^2(\Om)} + \nrm{|z|\md{\partial_tu}}_{L^2(\Om)}  \\&+\sum_{1\leq i,j \leq N} \nrm{X_iX_ju}_{L^2(\Om)} + \nrm{X_iY_ju}_{L^2(\Om)} +\nrm{Y_iX_j u}_{L^2(\Om)}+\nrm{Y_iY_j u}_{L^2(\Om)}.\notag
 \end{align}

  Also, from the Sobolev embedding theorem as in \cite{Xu}, it follows that
 \begin{equation}\label{hgrInt}
  u, \nabla_H u \in L^{2^*}_{loc} ( \Om),\ 2^*= \frac{2(N+2)}{N}
 \end{equation}
 where
 \[
 \nabla_H u = ( X_1 u,\dots, X_N u, Y_1 u,\dots Y_N u).
 \]

\section{Carleman Estimates}
In this section, we first recall a $L^{p}-L^{q}$ type Carleman estimate derived in \cite{GarShen} using which we prove one of our strong unique continuation result as in Theorem \ref{sucpGr} below. This  corresponds to the situation when the potential $V$ has the growth assumption as in \eqref{conPot}.   We then   derive a particular $L^2-L^2$ Carleman inequality where we obtain a certain asymptotic behaviour of the constant involved in the inequality   in terms of the parameter  $s$, where $s$ corresponds to the exponent of the singular weight in the Carleman estimate ( See Theorem \ref{L2L2CarEst} below).  Using such an estimate, we  argue as in \cite{Pa} and obtain  strong unique continuation property for \eqref{e0} when the potential $V$ satisfies the Hardy type growth assumption as in \eqref{conPot1}. The reader should note that our proof of this  new $L^{2}-L^{2}$  Carleman estimate   relies crucially on the following $L^{2}-L^{2}$ estimate for the projection operator $P_k$   established  in \cite{GarShen} which can be  stated as follows.

\begin{thm}\label{wtdL2est}
There exists a constant $C>0$ depending only on $N$ and $\alpha$, such that for any $h\in L^2(\Om,d\Om)$  
\begin{align}\label{wtdL2ineq}
\int_\Om \md{\sin^{-\alpha} \phi P_k\bct{\sin^{-\alpha}(.)h}(\phi,\omega)}^2 d\Om \leq C \int_\Om |h|^2 d\Om.
\end{align}
Here, $0\leq \alpha <\frac{1}{2}$, if $N$ is even and $0\leq \alpha <\frac{3}{8}$, if $N$ is odd.
\end{thm}
Next we recall a Carleman inequality derived in \cite{GarShen} ( see Theorem 5.1 in \cite{GarShen}), which will be instrumental in proving Theorem \ref{sucpGr}.  For the rest of the discussion in this paper,  we will denote by  $\mcal{L}$ the  Baouendi-Grushin operator on $\ro^{N+1}$  as in  \eqref{e0} defined by

\begin{align}\label{grushin}
\mcal{L} := \Delta_z +|z|^2\frac{\partial^2}{\partial t^2}.
\end{align}

\begin{thm}\label{LpLqCarEst}
Let $0<\delta<\frac{1}{4}$, $s>100$ and $\operatorname{dist}(s,\mathbb{N})=\frac{1}{2}$ Suppose that $p=\frac{2N}{N-1}$ and $q = \frac{2N}{N+1}$. Then there exists constant $C>0$ depending only on $ \delta$ and $N$, such that for $f\in C_c^\infty\bct{\ro^{N+1}\setminus\{0\}}$, the following inequality holds

\begin{align}\label{eCarEst}
\nrm{\rho^{-s} \bct{\operatorname{sin} \phi}^\delta f}_{L^p\bct{\ro^{N+1}, \frac{dzdt}{\rho^{N+2}}}} \leq C \nrm{\rho^{-s+2} \bct{\operatorname{sin} \phi}^{-\delta} \mcal{L}(f)}_{L^q\bct{\ro^{N+1}, \frac{dzdt}{\rho^{N+2}}}},
\end{align}
if $N\geq 2$ is even, and 

\begin{align}\label{oCarEst}
\nrm{\rho^{-s} \bct{\operatorname{sin} \phi}^{\frac{1}{4p}+\delta} f}_{L^p\bct{\ro^{N+1}, \frac{dzdt}{\rho^{N+2}}}} \leq C \nrm{\rho^{-s+2} \bct{\operatorname{sin} \phi}^{-\frac{1}{4p}-\delta} \mcal{L}(f)}_{L^q\bct{\ro^{N+1}, \frac{dzdt}{\rho^{N+2}}}},
\end{align}
if $N\geq 3$ is odd. 
\end{thm}
Now with  certain modifications (which will be pointed out) in  the proof of Theorem 5.1 in \cite{GarShen}, we show  that the following $L^2-L^2$ Carleman inequality can be derived. As remarked earlier, the main feature of this inequality is the prescribed asymptotic behaviour of the constant and that  is crucially  used in  the proof of  our Theorem \ref{sucp2} which concerns the growth assumption on $V$ as in \eqref{conPot1}. 
\begin{thm}\label{L2L2CarEst}
Let $0<\delta<\frac{1}{4}$, $s>100$ and  $\operatorname{dist}(s,\mathbb{N})=\frac{1}{2}.$  
Then there exists constant $C>0$ depending only on $ \delta$ and $N$, such that for $f\in C_c^\infty\bct{\ro^{N+1}\setminus\{0\}}$, the following inequality holds, 

\begin{align}\label{eL2L2CarIneq}
\nrm{\rho^{-s} \bct{\operatorname{sin} \phi}^\delta f}_{L^2\bct{\ro^{N+1}, \frac{dzdt}{\rho^{N+2}}}} \leq\frac{C \log_2 s}{s} \nrm{\rho^{-s+2} \bct{\operatorname{sin} \phi}^{-\delta} \mcal{L}(f)}_{L^2\bct{\ro^{N+1}, \frac{dzdt}{\rho^{N+2}}}},
\end{align}
if $N\geq 2$ is even, and 

\begin{align}\label{oL2L2CarIneq}
\nrm{\rho^{-s} \bct{\operatorname{sin} \phi}^{\frac{1}{8}+\delta} f}_{L^2\bct{\ro^{N+1}, \frac{dzdt}{\rho^{N+2}}}} \leq \frac{C \log_2 s}{s}\nrm{\rho^{-s+2} \bct{\operatorname{sin} \phi}^{-\frac{1}{8}-\delta} \mcal{L}(f)}_{L^2\bct{\ro^{N+1}, \frac{dzdt}{\rho^{N+2}}}},
\end{align}
if $N\geq 3$ is odd.

\end{thm}
\begin{rem}
We note that for the $L^{2}-L^{2}$ Carleman estimate corresponding to  the standard Laplacian, i.e. $\mcal{L}=\Delta$, it can be shown that  the asymptotic behavior of the constant  is infact   $\frac{C}{s}$  as $s \to \infty$ for some universal $C$( see for instance \cite{ABV}). This is clearly better than the  one we have  in Theorem \ref{L2L2CarEst}  above.   However for our application  to unique continuation as in Theorem \ref{sucp2}, it turns out that the asymptotic behaviour of the constant that we obtain in Theorem \ref{L2L2CarEst} above suffices. 
\end{rem}

\begin{proof}
We will point out the changes only for the case  $N$ is even. In the case of $N$ odd, the proof will follow similarly. In view of the discussion   in the proof of Theorem 5.1 in  \cite{GarShen}(  more precisely, as on page 157-158 in \cite{GarShen}), it suffices to show that the following inequality holds
\begin{align}\label{potEst}
\nrm{R_s(g)}_{L^2\left(\ro\times\Om, \bct{\sin \phi}^{-1+2\delta} dyd\Om\right)} \leq  \frac{C \text{log}_2 s}{s} \nrm{g}_{L^2\bct{\ro\times \Om, \bct{\sin \phi}^{-1-2\delta}dyd\Om}}, \ \forall g \in C_c^\infty(\ro\times \Om).
\end{align}
Here,
\begin{align}\label{invMellin}
R_s(g)(y,\phi,\omega) = \int_{\ro} \int_{\ro} e^{i(y-x)\eta} \sum_{k=0}^\infty \frac{Q_k(g)(x,\phi,\omega)}{a_s(\eta,k)} d\eta dx, 
\end{align}
where  the operator $Q_k$ is defined by
\begin{align}\label{invProj}
Q_k(g)(x,\phi,\omega) = P_k\bct{\frac{g(x,\cdot,\cdot)}{\sin(\cdot)}}(\phi,\omega)
\end{align}

\medskip

and
\medskip

\begin{align}\label{mellinMultiplier}
a_s(\eta,k) &= -\bct{\eta -i \bct{\bct{s+\frac{N+1}{2}}-\sqrt{k(N+k)+s+\frac{(N+1)^2}{4}}}}\notag\\ &\cdot \bct{\eta -i \bct{\bct{s+\frac{N+1}{2}}+\sqrt{k(N+k)+s+\frac{(N+1)^2}{4}}}}.
\end{align}

Now, fix $s>100$ such that $dist(s,\mathbb{N})=\frac{1}{2}$. Let $m$ satisfies $2^m\leq \frac{s}{10} <2^{m+1}$. Similar to \cite{GarShen}, we choose a partition of unity $\{\Phi_\beta\}_{\beta=0}^m$ for $\ro_+$ such that
\begin{align}\label{parUty}
\begin{cases}
\sum_{\beta} \Phi_\beta(r) =1 , \text{ for all  }r>0 \\ \operatorname{supp}\Phi_\beta \subset  \{r: 2^{\beta-2} \leq r \leq 2^\beta\}, \ \beta =1,2,\dots,m-1 \\ \operatorname{supp} \Phi_0 \subset \{r:0<r\leq 1\} \\ \operatorname{supp} \Phi_m  \subset \{r: r\geq \frac{s}{40}\}
\end{cases}
\end{align}
and 
\begin{align}\label{parUty1}
\md{\frac{d^l}{dr^l} \Phi_\beta(r)} \leq \frac{C_l}{2^{\beta l}}, \ l=0,1,2,\dots
\end{align}
For $0\leq \beta \leq m$, we define 
\begin{align}\label{parUtyMelMul}
b_s^\beta(\eta, k) = \frac{1}{a_s(\eta,k)} \Phi_\beta \bct{\md{\eta -i \bct{\bct{s+\frac{N+1}{2}}-\sqrt{k(N+k)+s+\frac{(N+1)^2}{4}}}}}
\end{align}
and 
\begin{align}\label{parUtyInvMel}
R_s^\beta(g)(y,\phi,\omega) = \int_{\ro} \int_{\ro} e^{i(y-x)\eta} \sum_{k=0}^\infty Q_k(g)(x,\phi,\omega)b_s^\beta(\eta,k) d\eta dx.
\end{align}

We further define
\begin{align}\label{intrmdtFun}
F_s^\beta(x,y,\phi,\omega) = \int_{\ro} e^{i(y-x)\eta} \sum_{k=0}^\infty Q_k(g)(x,\phi,\omega)b_s^\beta(\eta,k) d\eta.
\end{align}
Now, suppose there exists $f_s^\beta \in L^1(\ro)$ such that 
\begin{align}\label{intrmdtIneq}
\nrm{F_s^\beta(x,y,\cdot,\cdot)}_{L^2\left(\Om,\bct{\sin \phi}^{-1+2\delta}d\Om\right)} \leq f_s^\beta(|x-y|)\nrm{g(x,\cdot,\cdot)}_{L^2\bct{\Om, \bct{\sin\phi}^{-1-2\delta}d\Om}},
\end{align}
then we can use Minkowski's integral inequality to estimate  $\nrm{R^\beta_s(g)}_{L^2\bct{\ro\times\Om, \bct{\sin\phi}^{-1+2\delta}dyd\Om}}$ in the following manner
\begin{align}\label{anal}
&\nrm{R^\beta_s(g)}_{L^2\bct{\ro\times\Om, \bct{\sin\phi}^{-1+2\delta}dyd\Om}}\\& =  \bct{\int_{\ro}\bct{\bct{\int_\Om \bct{\int_{\ro}F_s^\beta(x,y,\phi,\omega)\bct{\sin\phi}^\frac{-1+2\delta}{2}dx }^2 d\Om}^\frac{1}{2}}^2dy}^\frac{1}{2}\notag \\ &\leq \bct{\int_{\ro} \bct{\int_{\ro} \nrm{F_s^\beta(x,y,\cdot,\cdot)}_{L^2\bct{\Om,\bct{\sin\phi}^{-1+2\delta}d\Om}}dx}^2dy}^\frac{1}{2}\notag \\ &\leq  \bct{\int_\ro \bct{\int_{\ro}f_s^\beta(|x-y|) \nrm{g(x,\cdot,\cdot)}_{L^2\bct{\Om,\bct{\sin\phi}^{-1-2\delta}d\Om}}dx}^2dy}^\frac{1}{2}\notag \\&\leq \nrm{f_s^\beta}_{L^1(\ro)} \nrm{g}_{L^2\bct{\ro\times\Om,\bct{\sin\phi}^{-1-2\delta}dyd\Om}}\notag
\end{align}
where to get the last two inequality we have used \eqref{intrmdtIneq} and the  Young's inequality for convolution respectively.  Now since 
\[
R_s= \sum_{\beta=0}^m R_s^\beta,
\]

 therefore the  preceding inequality clearly shows that to establish \eqref{potEst} we only need to find appropriate $f_s^\beta$ in \eqref{intrmdtIneq}. 
\par First consider $0\leq \beta \leq m-1$. In this case, if $b_s^\beta(\eta,k)\neq 0$, then by \eqref{parUty},
\begin{align*}
\delta 2^{\beta-2} \leq \md{\eta -i \bct{\bct{s+\frac{N+1}{2}}-\sqrt{k(N+k)+s+\frac{(N+1)^2}{4}}}} \leq 2^\beta.
\end{align*}
Hence, $|\eta| \leq 2^\beta$ and $|s-k|\leq 2^{\beta+1}$. Therefore, there are at most $2^{\beta+2}$ nonzero terms in the sum over $k$ in \eqref{parUtyMelMul} and one can compare the values of these $k'$s to $s$. So as in (5.19) in \cite{GarShen}, using \eqref{parUty1}, \eqref{mellinMultiplier} and \eqref{parUtyMelMul} we conclude that 
\begin{align}\label{bds1Der}
\md{\frac{\partial^j}{\partial \eta^j} b_s^\beta(\eta,k)} \leq \frac{C_j2^{-\beta}}{s2^{j\beta}}.
\end{align}
Invoking \eqref{wtdL2ineq} with $h= (\text{sin}\ \phi)^{\alpha-1}  g$, we get
\begin{align*}
\nrm{Q_k(g)}_{L^2\bct{\Om, \bct{\sin \phi}^{-2\alpha}d\Om}} \leq C \nrm{g}_{L^2\bct{\Om, \bct{\sin \phi}^{2\alpha-2}d\Om}},
\end{align*}
for $0<\alpha<\frac{1}{2}$. Choosing, $\delta = \frac{1}{2}-\alpha$ in the last inequality we get
\begin{align}\label{invProjEst}
\nrm{Q_k(g)}_{L^2\bct{\Om, \bct{\sin \phi}^{-1+2\delta}d\Om}} \leq C \nrm{g}_{L^2\bct{\Om, \bct{\sin \phi}^{-1-2\delta}d\Om}}
\end{align}
Now performing integration by parts and using \eqref{bds1Der}, \eqref{invProjEst}  we conclude that
\begin{align*}
&\nrm{F^\beta_s(x,y,\cdot,\cdot)}_{L^2\bct{\Om, \bct{\sin \phi}^{-1+2\delta}d\Om}} \\&\leq \frac{C}{|y-x|^j} \sum_{k=0}^\infty \int_{\ro} \md{\left(\frac{\partial}{\partial\eta}\right)^jb_s^\beta(\eta,k)} d\eta \nrm{Q_k(g)}_{L^2\bct{\Om, \bct{\sin \phi}^{-1+2\delta}d\Om}} \\&\leq \frac{C2^\beta}{s(2^\beta|y-x|)^j} \nrm{g}_{L^2\bct{\Om, \bct{\sin \phi}^{-1-2\delta}d\Om}}.
\end{align*}
In the above estimate, we   crucially used the fact that the support of integral lies in $\{|\eta| \leq \beta\}$ and the fact that atmost $2^{\beta+2}$ terms  survive in the above summation over $k$.

\medskip

If we  now choose $j=10$ and $j=0$, then  we obtain
\begin{align}
\nrm{F^\beta_s(x,y,\cdot,\cdot)}_{L^2\bct{\Om, \bct{\sin \phi}^{-1+2\delta}d\Om}} \leq \frac{C}{s\bct{1+2^\beta|y-x|}^{10}} 2^\beta \nrm{g}_{L^2\bct{\Om, \bct{\sin \phi}^{-1-2\delta}d\Om}}.\notag
\end{align}
Thus we are in the form of \eqref{intrmdtIneq} with $f_s^\beta(r) = \frac{C}{s\bct{1+2^\beta|r|}^{10}} 2^\beta$. Clearly, $\nrm{f_s^\beta}_{L^1(\ro)} \leq \frac{C}{s}$. Then from the estimates as in \eqref{anal}, this   implies for each $\beta=0, ... m-1$
\begin{equation}
\nrm{R_s^\beta(g)}_{L^2\left(\ro\times\Om, \bct{\sin \phi}^{-1+2\delta} dyd\Om\right)}\leq \frac{C }{s} \nrm{g}_{L^2\bct{\ro\times \Om, \bct{\sin \phi}^{-1-2\delta}dyd\Om}}.
\end{equation}

Therefore by  summing over $\beta=0$ to $\beta=m-1$ and by  using the fact that
\[
m \leq \text{log}_{2} s
\]

 we obtain
\begin{align}\label{finPotEst1}
& \sum_{\beta=0}^{m-1} \nrm{R_s^\beta(g)}_{L^2\left(\ro\times\Om, \bct{\sin \phi}^{-1+2\delta} dyd\Om\right)}\leq  \frac{C \text{log}_2 s}{s} \nrm{g}_{L^2\bct{\ro\times \Om, \bct{\sin \phi}^{-1-2\delta}dyd\Om}}.
\end{align}
Finally we  consider the case of $\beta =m$. For this, we observe that on the support of $b_s^m(\eta,k)$ we have 
\begin{equation}\label{eq}
\md{a_s(\eta,k)} \thicksim \bct{|\eta|+s+k}^2 
\end{equation}
and also 
\begin{equation}\label{endCaseBds1Der}
\md{\bct{\frac{\partial}{\partial\eta}}^j b_s^N(\eta,k)} \leq \frac{C_j}{\bct{|\eta|^2+k+s}^{j+2}}.
\end{equation}
Similarly as in the previous case, by integration by parts, using  and \eqref{eq} and \eqref{endCaseBds1Der} and by choosing $j=0$ and $j= 2$  we arrive at

\begin{align}\label{endCaseFinPotEst}
\nrm{F^m_s(x,y,\cdot,\cdot)}_{L^2\bct{\Om, \bct{\sin \phi}^{-1-2\ep}d\Om}} \leq \sum_{k=0}^\infty\frac{C}{(k+s)\bct{1+(k+s)^2|y-x|^2}} \nrm{g}_{L^2\bct{\Om, \bct{\sin \phi}^{-1-2\ep}d\Om}}.
\end{align}
Thus we are again  in the situation as in  \eqref{intrmdtIneq} with $f_s^\beta(r) =\sum_{k=0}^\infty\frac{C}{(k+s)\bct{1+(k+s)^2|r|^2}}$. Clearly,
\begin{equation}
\nrm{f_s^\beta}_{L^1(\ro)} \leq C\sum_{k=0}^\infty \frac{1}{(k+s)^2} \leq \frac{C}{s}. \notag
\end{equation}
Therefore we conclude that
\begin{equation}\label{finPotEst2}
\nrm{R_s^m(g)}_{L^2\left(\ro\times\Om, \bct{\sin \phi}^{-1+2\delta} dyd\Om\right)} \leq  \frac{C}{s} \nrm{g}_{L^2\bct{\ro\times \Om, \bct{\sin \phi}^{-1-2\delta}dyd\Om}}
\end{equation}
Finally, combining \eqref{finPotEst1} and \eqref{finPotEst2} we obtain \eqref{potEst}. This completes the proof.
\end{proof}

\section{Strong Unique Continuation}
In this  section, we establish the strong unique continuation property  for \eqref{e0} with the growth assumptions on $V$ as in \eqref{growth} or \eqref{gr2} using the Carleman estimates  as stated in  Theorem \ref{LpLqCarEst} and Theorem \ref{L2L2CarEst}.  For $r>0$ and $t_0\in \ro$, we define 

\begin{align}\label{gaugeball}
B_r\bct{(0,t_0)} := \{(z,t)\in \ro^{N+1}: \bct{|z|^4+4|t-t_0|^2}^\frac{1}{4} < r\}, \  \ B_r:= B_r((0,0)).  
\end{align}


 
 \begin{defi}\label{iv}
  We recall that $u$ vanishes to infinite order at the point $(0,t_0)$ in the $L^p$ mean, if 
 \begin{align}\label{vanLp}
 \int_{B_r(0,t_0)} |u|^p dz dt = \mcal{O}(r^l), \text{ as } r \rightarrow 0 \text{ for all } l>0.
 \end{align}
 \end{defi}
 
 We note that in \cite{G} as well as in \cite{GLa}, the authors insisted  in their definition of vanishing to infinite order at $(0,0)$ that  the function $u$ must satisfy  the following weaker assumption
 \begin{equation}\label{wk}
 \int_{B_r} u^2 \psi dzdt = O(r^{l}),\ \text{for all $l>0$ as $r \to 0^+$}
 \end{equation}
 
 It however turns out that  for functions in  $M^{2,2}(B_1)$,   vanishing to infinite order as in \eqref{vanLp} ( for $p=2$)  is in fact equivalent to vanishing to infinite order in the sense of \eqref{wk}. This is the content of the next lemma. 
 
 \begin{lem}\label{eq}
 Let $u \in M^{2,2}( B_1)$. Then  $u$ vanishes to infinite order in the $L^{2}$ mean at $(0,0)$ ( in the sense of Definition \ref{iv}) if and only if $u$ vanishes to infinite order in the sense of \eqref{wk}.
 \end{lem}
 \begin{proof}
 First we note that if  $u$ vanishes to infinite order in the $L^{2}$ mean,  since $\psi \leq 1$, 
 we have 
 \[
 \int_{B_r} u^2 \psi dzdt \leq \int_{B_r} u^2  dzdt
 \]
 and hence it follows that $u$  vanishes to infinite order in the sense of \eqref{wk}. Now let us  look at the  converse implication. First we note that from \eqref{hgrInt} we have 
 \[
 u \in L^{2^*}_{loc} (B_1),\  \text{where}\ 2^*=\frac{2(N+2)} {N}
 \]
 Now if $u$ vanishes to infinite order in the sense of \eqref{wk}, it follows  from the interpolation inequality 
 \[
\nrm{u}_{L^q(B_r,\ \psi dz dt)} \leq \ \nrm{u}_{L^2(B_r,\ \psi dzdt)}^\theta \nrm{u}_{L^{2^*}(B_r, \ \psi dzdt)}^{1-\theta},\  \text{ where $\frac{\theta}{2} +\frac{1-\theta}{2^*}=\frac{1}{q}$}
\] 
that  for all $q < 2^{*}$,  we have 
\begin{equation}\label{v1}
\int_{B_r} u^{q} \psi dz dt= O(r^{l}),\ \text{ as $r \to 0^+$}
\end{equation}
for all $l>0$. 
Now we fix some $q \in (2, 2^*)$. Then  from the H\"older inequality it follows 
\begin{equation}\label{v2}
\int_{B_r} |u| dzdt \leq  (\int_{B_r} |u|^q \psi dzdt)^{1/q}  (\int_{B_r} \psi^{-\frac{1}{q-1}} dzdt)^{\frac{q-1}{q}}
\end{equation}
Now since $q>2$, therefore we have that $\frac{1}{q-1} < 1$ and hence it follows from the polar decomposition as in \eqref{pmeasure}  that 
\[
\int_{B_r} \psi^{-\frac{1}{q-1}} dzdt < \infty,\  \text{Note that $\psi=\sin \phi$}
\]
Therefore combining \eqref{v1} and \eqref{v2} we get that
\begin{equation}\label{v4}
\int_{B_r} |u| dzdt= O(r^k),\ \text{as $r \to 0$}
\end{equation}
for all $k>0$. Now again by using  an interpolation inequality of the type
\[
\nrm{u}_{L^2(B_r,  dz dt)} \leq \ \nrm{u}_{L^1(B_r,  dzdt)}^\theta \nrm{u}_{L^{2^*}(B_r,  dzdt)}^{1-\theta},\  \text{ where $\theta  +\frac{1-\theta}{2^*}=\frac{1}{2}$}
\] 
it follows from \eqref{v4}  that $u$ vanishes to infinite order in the $L^{2}$ mean as in Definition \ref{iv}. The claim in the lemma thus follows. 
 \end{proof}

Our first unique continuation result can  now be stated as follows.

\begin{thm}\label{sucpGr}
With $\mathcal{L}$ as in \eqref{grushin}, let  $u \in M^{2,2}(B_{r_0})$ for some $r_0>0$  be a solution to 
\begin{align}\label{sucpBnd}
\mathcal{L}u = Vu \ \text{ in } B_{r_0},
\end{align}
for a potential $V$ satisfying \eqref{conPot}. If $u$ vanishes to infinite order at $(0,0)$ in the sense of \eqref{wk},  then $u \equiv 0$ in $B_{r_0}$.
\end{thm}

\begin{proof}
The proof of this result uses  the  Carleman estimates as    in Theorem \ref{LpLqCarEst}.  First we note that in view of Lemma \ref{eq},  we have   that $u$ vanishes to infinite order in the $L^{2}$ mean  at $(0,0)$. Now as in \cite{GarShen}, we let $\xi \in C_c^\infty(\ro^{N+1})$ such that $\xi = 1$ when $\rho(z,t) \leq \frac{1}{2}$ and $\xi = 0$ when $\rho(z,t)\geq \frac{3}{4}$. Also,  we define $\Psi_j(\rho) = \Psi(j\rho)$, where $\Psi = 1-\xi$.  

\medskip

 We first consider the case when  $N$  is even. Without loss of generality, we may assume for simplicity that $r_0=1$. Using a  standard limiting argument (i.e. by approximation with smooth functions), one can show that the  Carleman estimate \eqref{eCarEst} holds for   $f= \xi\Psi_ju$. This  gives 

\begin{align}\label{CarEstsucpGr}
\nrm{\rho^{-s} \bct{\sin\phi}^\delta \xi \Psi_ju}_{L^p\bct{\ro^{N+1}, \frac{dzdt}{\rho^{N+2}}}} &\leq C \nrm{\rho^{-s+2} \bct{\sin\phi}^{-\delta}  \mcal{L}(\Psi_ju)}_{L^q\bct{B_r, \frac{dzdt}{\rho^{N+2}}} } \notag\\ &+\nrm{\rho^{-s+2} \bct{\sin\phi}^{-\delta}  \mcal{L}(\xi u)}_{L^q\bct{\rho \geq r, \frac{dzdt}{\rho^{N+2}}} } \\ &:= I +II,
\end{align}
where $0<\frac{3}{4j}<r<\frac{1}{2}$ are constants to be chosen later. For $\delta$ small enough, by using H\"older inequality and \eqref{pmeasure} we conclude that 
\begin{align}\label{2ndCarEstsucpGr}
II \leq C r^{-s+2-\frac{N+2}{q}} \nrm{\mcal{L}(\xi u)}_{L^2(B_1)} \leq C r^{-s+2-\frac{N+2}{q}} \nrm{u}_{M^{2,2}(B_1)}.
\end{align}
 Next, we estimate $I$. For this we note that,
 \begin{align}\label{id1}
 \mcal{L}(\Psi_j u) = \mcal{L}(\Psi_j) u + 2\nabla_z\Psi_j\cdot \nabla_z u + 2|z|^2 \partial_t \Psi_j\cdot\partial_tu+\Psi_j \mcal{L}(u) .
 \end{align}
 Now we note that  the derivatives of $\Psi_j$  are supported in $B_{\frac{3}{4j}} - B_{\frac{1}{2j}}$  and satisfy the following bounds
 \begin{equation}\label{id2}
|\nabla \Psi_j|, |\nabla^2 \Psi_j| \leq C_0 j^{k}\ \text{for some $C_0$ universal  and some $k$  }
\end{equation}

Therefore,   by using \eqref{id1}, \eqref{id2}, the equation \eqref{sucpBnd} satisfied by $u$ and an application of H\"older inequality  for $\delta$ small enough,     we can  estimate I as follows
\begin{align}\label{1stCarEstsucpGr1}
&I \leq C \nrm{\rho^{-s+2} \bct{\sin \phi}^{-\delta} V \Psi_ju}_{L^q(B_r,\frac{dzdt}{\rho^{N+2}})} + Cj^M \bct{\int_{B_{\frac{3}{4j}}- B_{\frac{1}{2j}}}|u|^2dz dt}^\frac{1}{2}   \\ 
   &+ Cj^M \bct{\int_{B_{\frac{3}{4j}}- B_{\frac{1}{2j}}}\bct{|\nabla_zu|^2+|z|^2|\partial_tu|^2}dzdt}^\frac{1}{2}
   \notag
\end{align}
where $M>0$ is a constant which depends on $s$. Now  using the following variant of the Caccioppoli inequality  
\begin{equation}\label{ca}
\int_{B_{d} - B_{d/2}} ( |\nabla_z u|^2 +|z|^2 |\partial_t u|^2 )dz dt \leq \frac{C}{d^2}  \int_{B_{2d} - B_{d/4}}  (u^2 + |V| |u|^2) dz dt,\ d>0,
\end{equation}
 
 the vanishing to infinite order property  of $u$ and  the growth assumption on $V$ as in \eqref{conPot}, we can conclude that 
 \begin{equation}\label{z}
 Cj^M \bct{\int_{B_{\frac{3}{4j}}- B_{\frac{1}{2j}}}|u|^2dz dt}^\frac{1}{2}   
   + Cj^M \bct{\int_{B_{\frac{3}{4j}}- B_{\frac{1}{2j}}}\bct{|\nabla_zu|^2+|z|^2|\partial_tu|^2}dzdt}^\frac{1}{2}  \to 0\ \text{as}\ j \to \infty
   \end{equation}

Also, using H\"older inequality ( since $\frac{1}{p}= \frac{1}{q} - \frac{1}{N}$),  we can estimate the quantity 
\[
\nrm{\rho^{-s+2} \bct{\sin \phi}^{-\delta} V \Psi_ju}_{L^q(B_r,\frac{dzdt}{\rho^{N+2}})}
\]
 in \eqref{1stCarEstsucpGr1}
using the growth assumption on $V$ as in  \eqref{conPot} in the  following way,
\begin{align}\label{lqMainEst}
C &\nrm{\rho^{-s+2} \bct{\sin \phi}^{-\delta} V \Psi_ju}_{L^q(B_r,\frac{dzdt}{\rho^{N+2}})}\notag \\ &\leq C \nrm{\bct{\sin \phi }^{-2\delta} f(\rho)}_{L^N(B_r, \frac{dzdt}{\rho^{N+2}})}  \nrm{\rho^{-s} \bct{\sin \phi}^\delta \Psi_ju}_{L^p\bct{B_r,\frac{dzdt}{\rho^{N+2}}}}
\end{align}
Therefore, plugging \eqref{lqMainEst}  and \eqref{z} into \eqref{1stCarEstsucpGr1} and then also by using the estimate for II as in  \eqref{2ndCarEstsucpGr} we get from \eqref{CarEstsucpGr} (after passing $j\rightarrow \infty$) that the following inequality holds 

\begin{align}\label{finCarEst}
\nrm{\rho^{-s} \bct{\sin \phi}^\delta u}_{L^p\bct{B_r,\frac{dzdt}{\rho^{N+2}}}}   &\leq C \nrm{\bct{\sin \phi }^{-2\delta}  f(\rho)}_{L^N(B_r, \frac{dzdt}{\rho^{N+2}})} \nrm{\rho^{-s} \bct{\sin \phi}^\delta u}_{L^p\bct{B_r,\frac{dzdt}{\rho^{N+2}}}} \\& + C r^{-s+2-\frac{N+2}{q}} \nrm{u}_{M^{2,2}(B_1)}\notag.
\end{align}
Now using  $dzdt= \frac{1}{2}(\sin \phi)^{\frac{N-2}{2}} \rho^{N+1} d\rho d\phi d\omega$, we get

\begin{align}\label{sm1}
& \nrm{\bct{\sin \phi }^{-2\delta} f(\rho)}_{L^N(B_r, \frac{dzdt}{\rho^{N+2}})} = \left(\frac{1}{2} \int_{0}^r \frac{f(\rho)^N}{\rho} d\rho \int_{S^{N-1}} d\omega \int_{0}^{\pi}  (\sin \phi)^{-2 \delta N + \frac{N-2}{2}} d\phi\right)^{1/N}
\\
& \leq   \left(\frac{1}{2} \int_{0}^r \frac{f(\rho)}{\rho} d\rho \int_{S^{N-1}} d\omega \int_{0}^{\pi}  (\sin \phi)^{-2 \delta N + \frac{N-2}{2}} d\phi \right)^{1/N},\ \text{for $r$ small enough such that $f(r) \leq 1$}.
\notag
\end{align}

At this point using Dini integrability of $f$ as in  \eqref{Dini}, it follows from \eqref{sm1} that if we   choose $r, \delta $ small enough, then one can ensure that  

\[
C \nrm{\bct{\sin \phi }^{-2\delta} f(\rho)}_{L^N(B_r, \frac{dzdt}{\rho^{N+2}})} <\frac{1}{2}
\]
and therefore  the term $C\nrm{\bct{\sin \phi }^{-2\delta} f(\rho)}_{L^N(B_r, \frac{dzdt}{\rho^{N+2}})}\nrm{\rho^{-s} \bct{\sin \phi}^\delta u}_{L^p\bct{B_r,\frac{dzdt}{\rho^{N+2}}}}$  in \eqref{finCarEst} can be absorbed in the left hand side. Consequently we obtain 
\begin{align*}
\nrm{\bct{\frac{\rho}{r}}^{-s}\bct{\sin \phi}^\delta u}_{L^p\bct{B_r, \frac{dzdt}{\rho^{N+2}}}} \leq C \nrm{u}_{M^{2,2}(B_1)}.
\end{align*}
Now with  $s= k+\frac{1}{2}$ by letting  $k\rightarrow \infty$, we conclude $u\equiv 0$ in $B_r$. Outside of $B_r$, we have that $V$ is bounded and therefore we can now apply the results in  \cite{GarShen} to  conclude that $u \equiv 0$ in $B_1$. 

\medskip

 To handle the case when $N\geq 3$  is  odd,   we instead use the Carleman estimates as in \eqref{oCarEst} similar to that in the proof of Theorem 6.4 in \cite{GarShen}. Then by repeating the  same arguments as above, and the fact that 
 
 \begin{align}\label{essEstOdd}
&\nrm{\bct{\sin \phi}^{-\frac{1}{4p}-\delta}v}_{L^q(B_r)} \leq C \nrm{v}_{L^2(B_r)} \text{  for $\delta$ small  }
\end{align}

which follows from H\"older inequality, we obtain  by a limiting argument as before

\begin{align}\label{fin1}
& \nrm{\rho^{-s} \bct{\sin \phi}^{\frac{1}{4p}+ \delta} u}_{L^p\bct{B_r,\frac{dzdt}{\rho^{N+2}}}}   \\ &\leq C \nrm{\bct{\sin \phi }^{-\frac{1}{2p} -2\delta}  f(\rho)}_{L^N(B_r, \frac{dzdt}{\rho^{N+2}})} \nrm{\rho^{-s} \bct{\sin \phi}^{\frac{1}{4p} +\delta }u}_{L^p\bct{B_r,\frac{dzdt}{\rho^{N+2}}}} \notag\\ &+ C r^{-s+2-\frac{N+2}{q}} \nrm{u}_{M^{2,2}(B_1)}\notag.
\end{align}
Now again by using polar coordinates, we observe that for $\delta$ small enough 
\begin{align}\label{fin2}
\nrm{\bct{\sin \phi }^{-\frac{1}{2p} -2\delta}  f(\rho)}_{L^N(B_r, \frac{dzdt}{\rho^{N+2}})}= o(1), \text{ as } r \rightarrow 0. 
\end{align}
Therefore from \eqref{fin1} and \eqref{fin2}, a careful imitation of the previous proof will yield 
\begin{align*}
\nrm{\rho^{-s} \bct{\sin \phi}^{\frac{1}{4p}+\delta} u}_{L^p\left(B_r,\frac{dzdt}{\rho^{N+2}}\right)} \leq C r^{-s+2-\frac{N+2}{q}} \nrm{u}_{M^{2,2}(B_1)}.
\end{align*}
From here we can use the same arguments as in the case of $N\geq 2$ even to conclude the result for the case of $N\geq 3$ odd. This completes the proof.

\end{proof}
The next unique continuation result concerns the Hardy type growth assumption for the potential $V$ as in \eqref{conPot1}.

\begin{thm}\label{sucp2}
With $\mathcal{L}$ as in \eqref{grushin}, let  $N$ be even, $r_0>0$ and $u \in M^{2,2}\bct{B_{r_0}}$  be a solution to 
\begin{align}\label{sucpBnd1}
\mathcal{L}u = Vu \ \text{ in } B_{r_0}, 
\end{align}
for a potential $V$ satisfying \eqref{conPot1} for some $\ep>0$. If $u$  vanishes to infinite order at $(0,0)$ in the sense of \eqref{wk}, then $u \equiv 0$ in $B_{r_0}$.
\end{thm}
\begin{proof}
As in the proof of Theorem \ref{sucpGr}, we  may assume $r_0=1$ and also that $u$ vanishes to infinite order at $(0,0)$ in the $L^{2}$ mean.  Note that the later fact follows from Lemma \ref{eq}. We  now proceed  as in the proof of Theorem \ref{sucpGr} and consider the same cutoff  functions $\Psi_j$ and $\xi$. Now  by  applying the Carleman estimate as in  \eqref{eL2L2CarIneq} with the  $f=\xi\Psi_ju$  and $\delta=\frac{\ep}{2}$ we obtain
\begin{align}\label{CarEstsucpGr1}
&\nrm{\rho^{-s} \bct{\sin \phi}^\frac{\ep}{2} \xi \Psi_j u}_{L^2\bct{\ro^{N+1},\frac{dzdt}{\rho^{N+2}}}} \notag\\ & \leq \frac{C\log_2s}{s} \nrm{\rho^{-s+2} \bct{\sin \phi}^{-\frac{\ep}{2}} \mcal{L}(\Psi_j u)}_{L^2\bct{B_r,\frac{dzdt}{\rho^{N+2}}}} \notag\\ &+ \frac{C\log_2s}{s} \nrm{\rho^{-s+2} \bct{\sin \phi}^{-\frac{\ep}{2}} \mcal{L}(\xi u)}_{L^2\bct{\rho \geq r,\frac{dzdt}{\rho^{N+2}}}} := I+II,
\end{align}
where $j$ is large enough and $r$ is small enough satisfying $0<\frac{3}{4j}<r<\frac{1}{2}$. Clearly,
\begin{align}\label{2ndCarEstsucpGr1}
II \leq \frac{C\log_2s}{s} r^{-s+2} \nrm{\bct{\sin \phi}^{-\frac{\ep}{2}} \mcal{L}(\xi u)}_{L^2\bct{\rho\geq r,\frac{dzdt}{\rho^{N+2}}}}.
\end{align}
Observe that in the preceding inequality, the term involving the $L^2$ norm is finite, i.e.

\begin{equation}\label{f10}
 \nrm{\bct{\sin \phi}^{-\frac{\ep}{2}} \mcal{L}(\xi u)}_{L^2\bct{\rho\geq r,\frac{dzdt}{\rho^{N+2}}}} < \infty.
 \end{equation}
 This can be seen as follows. We first note that away from the origin, we have that $V$ is bounded. Therefore since $u$ satisfies \eqref{sucpBnd1},  it follows from the De Giorgi-Nash-Moser theory in this setting ( see for instance \cite{CDG})  that   $u$ is bounded away from the origin   and  consequently  the same holds for $\mcal{L}(u)$  as well  because of \eqref{sucpBnd1}. Therefore, using the boundedness of $u, \mcal{L}(u)$, higher integrability of $ \nabla_z{u}, |z|\partial_tu$ (as in  \eqref{hgrInt}), H\"older inequality and the fact that 

\begin{align*}
\int_{\{\rho \geq r\}} \bct{\sin \phi}^{-\delta} dzdt < \infty,
\end{align*}
for any $\delta \in (0,1)$, we can assert that the term involving $L^2$ norm in \eqref{2ndCarEstsucpGr1} is finite for small enough $\ep$. Note that in the growth assumption \eqref{conPot1}, since $\psi=\sin \phi  \leq 1$, therefore in the very first place we can assume that \eqref{conPot1} holds for $\ep$ small enough. 
\par 
Next we estimate $I$.  Again from the following interpolation inequality 
\[
\nrm{u}_{L^q} \leq \ \nrm{u}_{L^2}^\theta \nrm{u}_{L^{2^*}}^{1-\theta},\ \text{ where $\frac{\theta}{2} +\frac{1-\theta}{2^*}=\frac{1}{q}$} 
\]
  and also by using \eqref{hgrInt} we can infer  that since $u$ vanishes to infinite order at $(0,0)$ in the $L^2$ mean, it also vanishes of infinite order in the $L^q$ mean for any $q\in (2,2^*)$. Similarly by using the variant of the Caccioppoli inequality as in \eqref{ca} and interpolation inequality as above,  we can assert  $|\nabla_z u|$ and $|z||\partial_t u|$ vanishes to infinite order at $(0,0)$ in the $L^q$ mean for any $q\in (2,2^*)$. Therefore, combining these with the fact that 
\[
\int_{B_1} \bct{\sin \phi}^{-\delta} dzdt <\infty, \ \forall \delta \in (0,1),
\]
we can conclude  that  $\bct{\sin \phi}^{-\ep}|u|^2$ and $\bct{\sin \phi}^{-\ep}\bct{ \md{\nabla_zu}^2+|z|^2\md{\partial_tu}^2}$ vanishes to infinite order in the $L^1$ mean at  the point $(0,0)$ for small enough $\ep$. Hence, for some large enough $M$  depending on $s$ we have, 

\begin{align}\label{1stCarEstsucpGr11}
I  &\leq \frac{C \log_2s}{s} \nrm{\rho^{-s+2} \bct{\sin \phi}^{-\frac{\ep}{2}} V\Psi_j u}_{L^2\bct{B_r,\frac{dzdt}{\rho^{N+2}}}} +\frac{C\log_2s}{s} j^M \bct{\int_{B_{\frac{3}{4j}-\frac{1}{2j}}} \bct{\sin \phi}^{-\ep} |u|^2 dz dt}^\frac{1}{2} \notag \\ &+ \frac{C\log_2s}{s}  j^M \bct{\int_{B_{\frac{3}{4j}- \frac{1}{2j}}} \bct{\sin \phi}^{-\ep}  \bct{ \md{\nabla_zu}^2+|z|^2\md{\partial_tu}^2}dz dt}^\frac{1}{2} \notag \\ &\leq \frac{C \log_2s}{s} \nrm{\rho^{-s} \bct{\sin \phi}^{\frac{\ep}{2}} \Psi_j u}_{L^2\bct{B_r,\frac{dzdt}{\rho^{N+2}}}} + o(1), \text{ as } j \rightarrow \infty.
\end{align}
To achieve the last inequality we have used the imposed condition \eqref{conPot1} on $V$. Therefore, letting  $j \rightarrow \infty$   and making use of inequalities \eqref{2ndCarEstsucpGr1}, \eqref{1stCarEstsucpGr11} we conclude from  \eqref{CarEstsucpGr1} that the following holds
\begin{align*}
\nrm{\rho^{-s} \bct{\sin \phi}^{\frac{\ep}{2}} u}_{L^2\bct{B_r,\frac{dzdt}{\rho^{N+2}}}} &\leq\frac{C \log_2s}{s} \nrm{\rho^{-s} \bct{\sin \phi}^{\frac{\ep}{2}} u}_{L^2\bct{B_r,\frac{dzdt}{\rho^{N+2}}}} \notag\\&+ \frac{C\log_2s}{s}r^{-s+2} \nrm{\bct{\sin \phi}^{-\frac{\ep}{2}} \mcal{L}(\xi u)}_{L^2\bct{\rho \geq r,\frac{dzdt}{\rho^{N+2}}}}.
\end{align*}
Now  letting $s = s_k: = k+\frac{1}{2}$, for large enough $k$ we have 
\[
 \frac{C\log_2 s_k}{s_k} < \frac{1}{2}
\]
and therefore in the preceding inequality, for large enough $k$, the first term can be absorbed in the left hand side and consequently we will have

\begin{align}\label{2ndLstSucp}
\nrm{\bct{\frac{\rho}{r}}^{-s_k} \bct{\sin \phi}^{\frac{\ep}{2}} u}_{L^2\bct{B_r,\frac{dzdt}{\rho^{N+2}}}} \leq \frac{C \log_2s_k}{s_k} r^2 \nrm{\bct{\sin \phi}^{-\frac{\ep}{2}} \mcal{L}(\xi u)}_{L^2\bct{\rho \geq r,\frac{dzdt}{\rho^{N+2}}}}.
\end{align}
 Now  in view of \eqref{f10}, by  passing to limit $k \rightarrow \infty$ in \eqref{2ndLstSucp} we can assert that $u\equiv 0$ in $B_r$.  The rest of the argument remains the same as in the proof of the previous theorem and hence we can again conclude that $u \equiv 0$ in $B_1$. 

\end{proof}

Now in the case when $N\geq 3$ is odd, because the nature of the Carleman estimate in \eqref{oCarEst} is different, therefore we cannot assert that Theorem \ref{sucp2} holds. However using  the estimate in  \eqref{oCarEst} and a slight adaptation of the previous proof, we can assert that the following  unique continuation property holds.

\begin{thm}\label{sucp4}
With $\mathcal{L}$ as in \eqref{grushin}, let  $N\geq 3$ be odd, $r_0>0$ and $u \in M^{2,2}\bct{B_{r_0}}$  be a solution to 
\begin{align}\label{st}
\mathcal{L}u= Vu \ \text{ in } B_{r_0}, 
\end{align}
for a potential $V$ satisfying 
\[
|V(z,t)| \leq \frac{C (\sin \phi)^{\frac{1}{4}+ \ep}}{\rho(z,t)^2} 
\]

for some $\ep>0$.  If $u$  vanishes to infinite order at $(0,0)$ in the  sense of \eqref{wk}, then $u \equiv 0$ in $B_{r_0}$.
\end{thm}

\section{Application to unique continuation on the Heisenberg group $\mathbb{H}^n$}
We recall that on the Heisenberg group $\mathbb{H}^{n}= \mathbb{R}^{2n+1}$,  if   we denote an arbitrary point by  $(x, y, t)= (x_1, ...x_n, y_1, ...y_n, t)$, then   the group operation is defined as  follows
\begin{equation}\notag
(x, y, t) \circ (x', y', t')= (x+x', y+y', t+ t'  + 2(x'.y- x.y') )
\end{equation}
Let $z=(x,y)$. Note that the horizontal vector space   $V_1$ is spanned by 
\begin{align}\label{d1}
 X_i&=\partial_{x_i} +2y_i \partial_t,\ i=1, .., n
\\
X_{n+j}&= \partial_{y_j} - 2x_j \partial_t,\  j=1, .., n
\notag
\end{align}
and the vertical space  $V_2$ is spanned by $\partial_t$.  Infact, $\mathbb{H}^n$ is a protypical  example of Carnot group of step 2. 
As is well known, the following  Hormander bracket generating  condition holds, 
\begin{equation}\notag
[X_i, X_{n+j}]= -4 \delta_{ij} \partial_t, \ \forall i,j \in \{1,\dots,n\}
\end{equation}
and consequently the  sub-Laplacian  defined by
\[
\Delta_\mathbb{H} u= \sum_{i=1}^{2n} X_i^2u 
\]
is hypoelliptic.  Note that in real coordinates we have, 

 \begin{equation}
  \Delta_{\mathbb{H}}u = \Delta_z u + \frac{|z|^2}{4} \partial_{t}^2 u  + \partial_t Tu
  \end{equation}
where
\[
Tu= \sum_{i=1}^n (y_j \partial_{x_j} u - x_j \partial_{y_j}u )
\]

As mentioned in the introduction, $Tu=0$ if and only if \eqref{t} holds and in which case  $\Delta_{\mathbb{H}} u$ is given by 

\begin{equation}\label{i10}
  \Delta_{\mathbb{H}}u = \Delta_z u + \frac{|z|^2}{4} \partial_{t}^2 u  
  \end{equation}
  
  Now  we recall that  in \cite{GLa}, Garofalo and Lanconelli  showed that the following unique continuation results hold.
  \begin{thm}[GL]\label{gl}
  Let $u$ be a solution to 
  \[
  \Delta_{\mathbb{H}}u= Vu
  \]
  such that
  \[
  |t Tu(z,t)| \leq g(\rho(z,t)) |z|^2 |u(z,t)|
  \]  
  
  for some Dini integrable $g$ satisfying \eqref{Dini}.  Now corresponding to $N=2n$ and $\beta=m=1$,  if $V$ satisfies the growth  condition as in \eqref{growth},  then $u$ satisfies the strong unique continuation property at $0$. 
  
  \medskip
  
   If instead $V$ satisfies the growth condition  as in \eqref{gr2} for some $\delta$ small enough, then  there exists $r_0>0$ such that if 
   \[
   \int_{B_r} u^2 \psi= O( exp( -A r^{-\alpha})),\   \text{as $r \to 0^+$}
   \]
   for some $A, \alpha>0$, then $u \equiv 0$ in $B_{r_0}$.

  \end{thm}

  As previously mentioned in the introduction, the proof in \cite{GLa}  is based on Almgren type frequency function approach. Now in the situation of Theorem \ref{gl} when $Tu \equiv 0$, we obtain the following improvement of Theorem \ref{gl}  as a consequence of our unique continuation results Theorem \ref{sucpGr} and Theorem \ref{sucp2} in Section 4. 
  
  \begin{thm}\label{ap}
  Let $u$ be a solution to 
  \[
  \Delta_{\mathbb{H}}u= Vu,\  \text{in}\  B_{r_0}
  \]
 such that $u, X_i u, X_i u, X_iX_j u \in L^{2}(B_{r_0}(0,0))$, $\forall i,j \in \{1,\dots,2n\}$  and  $V$ satisfies either  the growth condition   as in \eqref{conPot} or as in \eqref{conPot1} corresponding to  $\beta=m=1$ and $N=2n$. Moreover assume that $Tu \equiv 0$. Then if $u$ vanishes to infinite order at $0$ in the sense of \eqref{wk}, then $u \equiv 0$ in $B_{r_0}$.
  \end{thm}
  
  \begin{proof}
  In view of the discussion around \eqref{i10} above, we note  that 
  \[
   \Delta_{\mathbb{H}}u = \Delta_z u + \frac{|z|^2}{4} \partial_{t}^2 u,\quad z=(x, y) \in \mathbb{R}^{2n}
   \]
   i.e. upto a normalization factor of $4$, we have that  $  \Delta_{\mathbb{H}}u= \mathcal{B}_1 u$.  Now if $V$ satisfies \eqref{conPot}, we can apply the result in Theorem \ref{sucpGr} to conclude that $u \equiv 0$.

   \medskip

   On the other hand, if $V$ instead  satisfies the growth condition in \eqref{conPot1}, we  apply the result in Theorem \ref{sucp2} to  again conclude that the  desired conclusion  holds. Note that in our case, we have that $N=2n$  and hence Theorem \ref{sucp2} can be applied which corresponds to the case when $N$ is even.  We would however like to direct the attention of the reader to  a subtle aspect in the application of Theorem \ref{sucp2}.  The reader should note that in the proof of Theorem \ref{sucp2}, we needed at an intermediate step that $ \md{\nabla_zu}^2+|z|^2\md{\partial_tu}^2 \in L^{1+\gamma}$ for some $\gamma >0$ in order to assert by an application of H\"older and interpolation type inequality that $\bct{\sin \phi}^{-\ep}\bct{ \md{\nabla_zu}^2+|z|^2\md{\partial_tu}^2}$ vanishes to infinite order in the $L^1$ mean at  the point $(0,0)$ for small enough $\ep$. That was also needed to ensure the finiteness of the quantity in \eqref{f10}.    In our situation, this is guaranteed  by the  fact that since $Tu=0$,  therefore we have
   \[
   \sum_{i=1}^{2n}X_i^2 u=  |\nabla_z u|^2 + |z|^2 (\partial_t u)^2
   \]
   and consequently by  the Folland-Stein embedding ( \cite{FS}),  the higher integrability of $ \md{\nabla_zu}^2+|z|^2\md{\partial_tu}^2$ follows.  The rest of the proof remains the same.   
      
     \end{proof}
 
 \def\cprime{$'$}





\end{document}